\documentclass[12pt]{amsart}%
\usepackage{amsmath,amsthm,amsfonts,amssymb}
\usepackage{verbatim}
\usepackage{latexsym}
\usepackage{marginnote}
\usepackage{amscd}
\usepackage{amssymb}
\usepackage{xcolor}
\usepackage{amsmath}
\usepackage{amsfonts}
\usepackage{graphicx}%
\providecommand{\U}[1]{\protect\rule{.1in}{.1in}}
\providecommand{\U}[1]{\protect\rule{.1in}{.1in}}
\providecommand{\U}[1]{\protect\rule{.1in}{.1in}}
\newtheorem{theorem}{Theorem}[section]
\newtheorem{lemma}[theorem]{Lemma}
\newtheorem{proposition}[theorem]{Proposition}
\newtheorem{definition}[theorem]{Definition}
\newtheorem{corollary}[theorem]{Corollary}

\newtheorem{question}[theorem]{Question}
\newtheorem{remark}[theorem]{Remark}

\theoremstyle{definition}

\begin{document}
\title[Groups with frames of translates]{Groups with frames of translates}
\author{Hartmut F\"uhr, Vignon Oussa}
\address{Lehrstuhl A f\"ur Mathematik \\
RWTH Aachen University, D-52056 Aachen\\
Germany\\
Department of Mathematics \\
Bridgewater State University\\
Bridgewater, MA 02324 \\
U.S.A.}
\email{vignon.oussa@bridgew.edu}

\begin{abstract}
Let $G$ be a locally compact group with left regular representation
$\lambda_{G}.$ We say that $G$ admits a frame of translates if there exist a
countable set $\Gamma\subset G$ and $\varphi\in L^{2}(G)$ such that
$(\lambda_{G}(x) \varphi)_{x \in\Gamma}$ is a frame for $L^{2}(G).$ The
present work aims to characterize locally compact groups having frames of
translates, and to this end, we derive necessary and/or sufficient conditions
for the existence of such frames. Additionally, we exhibit surprisingly large
classes of Lie groups admitting frames of translates.

\end{abstract}
\maketitle




\section{Introduction}

Throughout this paper, $G$ denotes a second countable locally compact group
$G$ with Haar measure $\mu_{G}$. We denote the associated $L^{2}$-space by
$L^{2}(G)$. $G$ acts on this space unitarily via the left regular
representation, which we denote by $\lambda_{G}$.

We briefly recall the definitions of \textbf{frames} and \textbf{Bessel
sequence}: A system $(\eta_{i})_{i\in I}$ of vectors in a Hilbert space is
called a \textbf{frame} if there exist constants $0<A\leq B<\infty$ such that,
for all $f\in\mathcal{H}$,
\[
A\Vert f\Vert^{2}\leq\sum_{i\in I}|\langle f,\eta_{i}\rangle|^{2}\leq B\Vert
f\Vert^{2}~.
\]
The constants $A,B$ are called frame bounds. We refer to $A$ as a lower frame
bound and to $B$ as an upper frame bound. If $(\eta_{i})_{i\in I}$ only admits
an upper frame bound then we say that $(\eta_{i})_{i\in I}$ is a
\textbf{Bessel family}.

We are interested in groups of the following type:

\begin{definition}
$G$ \textbf{admits a frame of translates}, or \textbf{is an FT group} if there
exists a family $\Gamma\subset G$ and $\varphi\in L^{2}(G)$ such that the
family $(\lambda_{G}(x) \varphi)_{x \in\Gamma} \subset L^{2}(G)$ is a frame;
i.e., there exist constants $0 < A \le B < \infty$ such that, for all $g \in
L^{2}(G)$,
\[
A \| g \|_{2}^{2} \le\sum_{x \in\Gamma} |\langle g, \lambda_{G}(x)
\varphi\rangle|^{2} \le B \| g \|_{2}^{2}~.
\]

\end{definition}

The next remark lists some known or expected results. Part (b) is
\cite[Theorem 1.2]{ChDeHe}, and the proof given in \cite{ChDeHe} relies on
fairly technical concepts such as Beurling density.

\begin{remark}
\label{facts}\text{ }\newline(a) If $G$ is discrete, then $(\lambda
_{G}(x)\delta_{e})_{x\in G}$ is an orthonormal basis of $L^{2}(G)=\ell^{2}
(G)$, where $\delta_{e}$ denotes the Kronecker delta at the neutral element.
Hence $G$ is an FT-group.\newline(b) $G=\mathbb{R}^{d}$ is not an $FT$ group.
\end{remark}

The main objective of the present work is to investigate solutions to the
following question

\begin{question}
\label{quest} Which non-discrete locally compact groups possess the FT property?
\end{question}

Intuitively, one may not expect many positive answers to Question \ref{quest}
outside the discrete case. However, our results will show that FT groups are
surprisingly not rare.

\begin{remark}
Instead of using subsets $\Gamma\subset G$, one may pose the central question
of our paper with reference to families $(x_{i})_{i\in I}\subset G$ and
associated families of vectors $(\lambda_{G}(x_{i})\varphi)_{i\in I}$. The
difference between the two is that the latter allows repetitions of elements.
However, it is easy to see that $(\lambda_{G}(x_{i})\varphi)_{i\in I}$ is a
frame if and only if $(\lambda_{G}(x)\varphi)_{x\in\Gamma}$ is a frame, where
$\Gamma=\{x_{i}:i\in I\},$ and additionally, $\sup_{y\in\Gamma}\sharp\{i\in
I:y=x_{i}\}<\infty.$ Hence exchanging subsets for families does not affect the
FT property, and we will freely switch between the two notions.
\end{remark}

In the course of this paper, we will freely use notions from frame theory,
representation theory of locally compact groups and Lie theory; our primary
references for these topics are \cite{MR1946982}, \cite{MR1397028}, and
\cite{Hilgert} respectively.

\section{Necessary criteria}

In this section, we will consider various necessary conditions for frames of
the type $(\lambda_{G}(x)\varphi)_{x\in\Gamma}$. These conditions will either
concern the family $\Gamma$ of shifts, the function $\varphi$, or the group.
In our analysis, we proceed precisely in this order.

\begin{definition}
Let $\Gamma$ be a subset of $G$. We say that $\Gamma$ is \textbf{(left)
relatively separated} if $\sup_{x\in G}\sharp\left(  \Gamma\cap xU\right)
<\infty$ for some and hence all relatively compact neighborhoods $U$ of the
identity. Next, we say that $\Gamma$ is $V$-\textbf{separated} if for some
relatively compact neighborhood $V$ of the identity, the family $(xV)_{x
\in\Gamma}$ consists of pairwise disjoint sets. $\Gamma$ is called
\textbf{separated} if it is $V$-separated for some suitable $V$.
\end{definition}

The following result has been rediscovered several times in frame theory. We
rephrase it for our setting.

\begin{lemma}
\label{separated} If $\left(  \lambda\left(  \gamma\right)  \varphi\right)
_{\gamma\in\Gamma}$ is a Bessel family then $\Gamma$ must be relatively separated.
\end{lemma}

\begin{proof}
Suppose that $\Gamma$ is not relatively separated. We consider the function
$x\mapsto\left\langle \varphi,\lambda\left(  x\right)  \varphi\right\rangle $
defined over $G.$ Since $x\mapsto\left\langle \varphi,\lambda\left(  x\right)
\varphi\right\rangle $ is continuous, there exists an open set $V$ around the
identity element such that $\inf\left\{  \left\vert \left\langle
\varphi,\lambda\left(  x\right)  \varphi\right\rangle \right\vert :x\in
V\right\}  =\mu>0.$ Next, for an arbitrary natural number $N,$ there exists
$y\in G$ such that $yV$ contains at least $N$ elements from $\Gamma.$ Next,
let $\Gamma_{N}=yV\cap\Gamma.$ As a result,
\begin{align*}
\sum_{\gamma\in\Gamma}\left\vert \left\langle \lambda\left(  y\right)
\varphi,\lambda\left(  \gamma\right)  \varphi\right\rangle \right\vert ^{2}
&  \geq\sum_{\gamma\in\Gamma_{N}}\left\vert \left\langle \lambda\left(
y\right)  \varphi,\lambda\left(  \gamma\right)  \varphi\right\rangle
\right\vert ^{2}\\
&  =\sum_{\gamma\in\Gamma_{N}}\left\vert \left\langle \varphi,\lambda\left(
y^{-1}\gamma\right)  \varphi\right\rangle \right\vert ^{2}\\
&  =\sum_{y\alpha\in\Gamma_{N}}\left\vert \left\langle \varphi,\lambda\left(
y^{-1}y\alpha\right)  \varphi\right\rangle \right\vert ^{2}\\
&  =\sum_{\alpha\in y^{-1}\Gamma_{N}}\left\vert \left\langle \varphi
,\lambda\left(  \alpha\right)  \varphi\right\rangle \right\vert ^{2}\\
&  \geq\sharp\left(  y^{-1}\Gamma_{N}\right)  \cdot\mu^{2}\\
&  \geq N\cdot\mu^{2}=\left(  \frac{N\cdot\mu^{2}}{\left\Vert \varphi
\right\Vert ^{2}}\right)  \left\Vert \lambda\left(  y\right)  \varphi
\right\Vert ^{2}.
\end{align*}
Now since $N$ is arbitrary, it follows that $\left(  \lambda\left(
\gamma\right)  \varphi\right)  _{\gamma\in\Gamma}$ is not a Bessel sequence.
\end{proof}

The following result is formulated in \cite[Lemma 3.3]{MR1021139}, with proof
attributed to \cite{MR809337}. Since we need it in the following, and the
argument in \cite{MR809337} is given for the slightly different context of
admissible coverings, we include a short proof.

\begin{lemma}
\label{uniform} Let $\Gamma\subset G$ denote a relatively separated set of a
locally compact group $G$. Then $\Gamma$ is the finite union of separated sets.
\end{lemma}

\begin{proof}
Fix a relatively compact and symmetric neighborhood $V \subset G$ of the
identity. Then relative separatedness of $\Gamma$ yields that
\[
\sup_{\gamma\in\Gamma} \sharp\{ \gamma^{\prime}\in\Gamma: \gamma^{\prime}V
\cap\gamma V \not = \emptyset\} \le\sup_{x \in G} \sharp\{ \gamma^{\prime}%
\in\Gamma: \gamma^{\prime} \in x V^{2} \} = m < \infty~.
\]
Zorn's Lemma allows to choose a subset $\Gamma_{1} \subset\Gamma$ that is
$V$-discrete and maximal with respect to inclusion. If $\Gamma_{1} = \Gamma$,
then $\Gamma$ itself is $V$-discrete. We continue this procedure of choosing a
maximal $V$-discrete $\Gamma_{s+1} \subset\Gamma\setminus\bigcup_{j \le s}
\Gamma_{j}$, as long as the complement is nonempty. We claim that this
procedure breaks off after at most $m+1$ steps. Assuming that $\gamma_{0}
\in\Gamma\setminus\bigcup_{j \le m+1} \Gamma_{j}$, we find for every fixed $1
\le j \le m+1$ that $\gamma_{0} \in\Gamma\setminus\bigcup_{i<j} \Gamma_{i}$,
and by maximality of $\Gamma_{j}$, $\Gamma_{j} \cup\{ \gamma_{0} \}$ is not
$V$-discrete. Hence there exists $\gamma_{j} \in\Gamma_{j}$ such that
$\gamma_{0} V \cap\gamma_{j} V \not = \emptyset$. Since the $\Gamma_{j}$ are
pairwise disjoint, this entails
\[
\sharp\{ \gamma^{\prime}\in\Gamma: \gamma^{\prime}V \cap\gamma_{0} V \not =
\emptyset\} \ge m+1 ~,
\]
contrary to our choice of $m$.
\end{proof}

We next derive necessary conditions on the function $\varphi$ giving rise to
frames of translates. Our aim is to show that for non-discrete groups such
functions must necessarily be somewhat pathological. For instance, bounded
functions with compact support will not do. In fact, we will be able to
exclude a substantially larger space of functions, namely a particular
\textbf{Wiener amalgam space}.

We define a local maximum function, as follows: Fix a compact neighborhood $U
$ of the identity. Given a measurable function $f$ on $G$, we define
\[
f_{U}^{\sharp}(x)=\mathrm{ess\,sup}_{y\in xU}|f(y)|.
\]
Next, given $1\leq p\leq\infty$, we define the Wiener amalgam spaces
$W(L^{\infty},L^{p})$ as the space of Borel functions $f$ for which the
respective norm
\[
\Vert f\Vert_{W(L^{\infty},L^{p})}=\Vert f_{U}^{\sharp}\Vert_{p}%
\]
is finite. It is well-known that, up to equivalence, the Wiener amalgam norm
does not depend on the choice of $U$.

Note that compactly supported, bounded functions $\varphi$ are contained in
$W(L^{\infty},L^{p})$, for all $1\leq p\leq\infty$. For the following result,
we also need the following. The convolution of two functions $f,\varphi$ on
$G$ is defined as the integral
\[
\left(  f\ast\varphi\right)  \left(  x\right)  =\int_{G}f\left(  y\right)
\varphi\left(  y^{-1}x\right)  dy
\]
and $\varphi^{\ast}(x)=\overline{\varphi(x^{-1})}$. Then a straightforward
calculation gives $\left(  f\ast\varphi^{\ast}\right)  \left(  x\right)
=\left\langle f,\lambda_{G}\left(  x\right)  \varphi\right\rangle ,$ for all
$f, \varphi\in L^{2}(G)$.

Now the next proposition excludes compactly supported bounded functions from
frame generation.

\begin{proposition}
\label{prop:phi_path} Let $G$ be non-discrete. Let $\varphi\in L^{2}(G)$ be
such that $\varphi^{\ast}\in W(L^{\infty},L^{2})$. Then there does not exist a
discrete set $\Gamma\subset G$ such that $(\lambda_{G}(x)\varphi)_{x\in\Gamma
})$ is a frame of $L^{2}(G)$.
\end{proposition}

\begin{proof}
Suppose by contradiction that $(\lambda_{G}(x)\varphi)_{x\in\Gamma})$ is a
frame with lower frame bound $A.$ Then, for all $f\in L^{2}(G)$, we have
\begin{equation}
\Vert f\Vert_{2}^{2}\leq A^{-1}\sum_{x\in\Gamma}|f\ast\varphi^{\ast}(x)|^{2}~.
\label{lb}%
\end{equation}
By Lemmas \ref{separated} and \ref{uniform}, we can write $\Gamma
=\bigcup_{i=1}^{n}\Gamma_{i}$ disjointly, and each $\Gamma_{i}$ is
$U$-discrete, for a suitable symmetric neighborhood $U$ of the identity. Hence
we get
\begin{align*}
\sum_{x\in\Gamma}|f\ast\varphi^{\ast}(x)|^{2}  &  =\sum_{i=1}^{n}\sum
_{x\in\Gamma_{i}}|f\ast\varphi^{\ast}(x)|^{2}\\
&  =\sum_{i=1}^{n}\sum_{x\in\Gamma_{i}}\frac{1}{|U|}\int_{xU}|f\ast
\varphi^{\ast}(x)|^{2}dy\\
&  \leq\sum_{i=1}^{n}\sum_{x\in\Gamma_{i}}\frac{1}{|U|}\int_{xU}\left\vert
\left(  f\ast\varphi^{\ast}\right)  _{U}^{\sharp}(y)\right\vert ^{2}dy\\
&  \leq\frac{n}{|U|}\Vert\left(  f\ast\varphi^{\ast}\right)  _{U}^{\sharp
}\Vert_{2}^{2}~.
\end{align*}
Applying the well-known pointwise estimate
\[
\left(  f\ast g\right)  _{U}^{\sharp}(y)\leq(|f|\ast g_{U}^{\sharp})(y)~,
\]
valid for arbitrary measurable functions $f,g$ yields
\[
\Vert\left(  f\ast\varphi^{\ast}\right)  _{U}^{\sharp}\Vert_{2}\leq\left\Vert
|f|\ast(\varphi^{\ast})_{U}^{\sharp}\right\Vert _{2}.
\]
Coming back to (\ref{lb}), we thus obtain%
\[
\Vert f\Vert_{2}\leq\left(  \frac{n}{A|U|}\right)  ^{1/2}\left\Vert
|f|\ast(\varphi^{\ast})_{U}^{\sharp}\right\Vert _{2}.
\]

On the other hand, Young's inequality yields for all $f \in L^{1}(G) \cap
L^{2}(G)$, that
\[
\left\|  | f | \ast(\varphi^{*})_{U}^{\sharp}\right\|  _{2} \le\| f \|_{1} \|
(\varphi^{*})_{U}^{\sharp}\|_{2} = \| f \|_{1} \|\varphi^{*} \|_{W(L^{\infty},
L^{2})}~,
\]
and the Wiener amalgam norm is finite by assumption.

In summary, we have shown
\begin{equation}
\Vert f\Vert_{2}\leq\left(  \frac{n}{A|U|}\right)  ^{1/2}\Vert\varphi^{\ast
}\Vert_{W(L^{\infty},L^{2})}\Vert f\Vert_{1}~, \label{estimate}%
\end{equation}
for all $f\in L^{1}(G)\cap L^{2}(G)$. Now replacing $f$ with the indicator
function of $U$ with $\lambda_{G}(U)\rightarrow0$, yields the desired contradiction.
\end{proof}

We next derive various classes of groups that are not FT.

\begin{corollary}
\label{compact_case}If $G$ is compact, then it is FT if and only if it is finite.
\end{corollary}

\begin{proof}
By Lemma \ref{separated} and the fact that relatively separated subsets of
compact groups are finite, $L^{2}(G)$ is finite-dimensional whenever $G$ is a
compact FT group.
\end{proof}

\begin{theorem}
\label{hom} Let $G$ be non-discrete, satisfying the following property: The
inverse of any subset of $G$ which is relatively separated is also relatively
separated. Then $G$ is not an FT group.
\end{theorem}

\begin{proof}
Suppose by ways of contradiction that the stated assumptions hold and that
there exists $\varphi\in L^{2}(G)$ and $\Gamma\subset G$ such that $\left(
\lambda\left(  \gamma\right)  \varphi\right)  _{\gamma\in\Gamma}$ is a frame
for $L^{2}(G).$ Thus, $\left(  \lambda\left(  \gamma\right)  \varphi\right)
_{\gamma\in\Gamma}$ is a Bessel sequence. By Lemma \ref{separated}, $\Gamma$
must be relatively separated. By assumption, $\Gamma^{-1}$ is also relatively
separated. Furthermore, according to Lemma \ref{uniform}, $\Gamma^{-1}$ can be
written as a disjoint union of separated sets. Let $\Gamma^{-1}=\bigcup
_{k=1}^{s}\Psi_{k}$ such that each collection $\left\{  xV:x\in\Psi
_{k}\right\}  $ consists of essentially disjoint subsets of $G.$ Then
\begin{align*}
\sum_{\gamma\in\Gamma}\left\vert \left\langle \chi_{V},\lambda\left(
\gamma\right)  \varphi\right\rangle \right\vert ^{2}  &  =\sum_{\gamma
\in\Gamma}\left\vert \left\langle \chi_{V},\chi_{V}\cdot\lambda\left(
\gamma\right)  \varphi\right\rangle \right\vert ^{2}\\
&  \leq\left\Vert \chi_{V}\right\Vert ^{2}\left(  {\sum\limits_{k=1}^{s}}%
\sum_{\gamma\in\Psi_{k}}\left\Vert \chi_{V}\cdot\lambda\left(  \gamma\right)
\varphi\right\Vert ^{2}\right) \\
&  =\left\Vert \chi_{V}\right\Vert ^{2}\left(  {\sum\limits_{k=1}^{s}}%
\sum_{\gamma\in\Psi_{k}}\int_{V}\left\vert \varphi\left(  \gamma^{-1}x\right)
\right\vert ^{2}dx\right) \\
&  =\left\Vert \chi_{V}\right\Vert ^{2}{\sum\limits_{k=1}^{s}}\left(
\int_{\bigcup_{\gamma\in\Psi_{k}}\gamma^{-1}V}\left\vert \varphi\left(
x\right)  \right\vert ^{2}dx\right)  .
\end{align*}
Note that since $\varphi$ is square-integrable,
\[
\int_{\bigcup_{\gamma\in\Psi_{k}}\gamma^{-1}V}\left\vert \varphi\left(
x\right)  \right\vert ^{2}dx\leq\int_{G}\left\vert \varphi\left(  x\right)
\right\vert ^{2}dx=\left\Vert \varphi\right\Vert ^{2}<\infty.
\]
By Lebesgue's Dominated Convergence Theorem, taking a nested family of
relatively compact and open sets converging to the singleton containing the
identity in $G,$ we obtain:
\[
\lim_{V\rightarrow\left\{  e\right\}  }\int_{\cup_{\gamma\in\Psi_{k}}%
\gamma^{-1}V}\left\vert \varphi\left(  x\right)  \right\vert ^{2}dx=0.
\]
Thus for any $\epsilon>0,$ there exists a sufficiently small open set $V$
around the identity such that
\[
\sum_{\gamma\in\Gamma}\left\vert \left\langle \chi_{V},\lambda\left(
\gamma\right)  \varphi\right\rangle \right\vert ^{2}\leq\epsilon\left\Vert
\chi_{V}\right\Vert ^{2}.
\]
This violates the lower frame bounds condition and gives us the desired contradiction.
\end{proof}

This observation allows to generalize \cite[Theorem 1.2]{ChDeHe} to a larger
class of groups called \textbf{[IN]-groups}: groups having a\textbf{\ compact
neighborhood} of the identity which is invariant under all inner automorphisms.

\begin{lemma}
\label{[IN]case} If $G$ is an [IN]-group then every relatively separated
subset of $G$ has a relatively separated inverse. In particular, nondiscrete
[IN]-groups are not FT.
\end{lemma}

\begin{proof}
Let $G$ be an [IN]-group. Then by definition, there exists a compact
neighborhood $W\subset G$ of the identity which is conjugation-invariant.
Next, let $\Gamma$ be a subset of $G$ which is relatively separated. By
assumption, $\sup_{x\in G}\sharp\left(  \Gamma\cap xW \right)  <\infty.$ On
the other hand for any $x\in G,$ note that
\[
\Gamma\cap xW =\Gamma\cap\left(  xW x^{-1}\right)  x = \Gamma\cap Wx~,
\]
and thus
\[
\left(  \Gamma\cap x W \right)  ^{-1} = \Gamma^{-1} \cap x^{-1} W^{-1}.
\]
Consequently, since inversion on $G$ is bijective,
\[
\sup_{y\in G}\sharp\left(  \Gamma^{-1}\cap y W^{-1}\right)  = \sup_{x \in G}
\sharp\left(  \Gamma^{-1}\cap x^{-1} W^{-1}\right)  = \sup_{x\in G}%
\sharp\left(  \Gamma\cap x W \right)  <\infty
\]
which proves that $\Gamma^{-1}$ is relatively separated.
\end{proof}

\begin{remark}
Clearly, Lemma \ref{[IN]case} applies to abelian groups, thus it directly
generalizes \cite[Theorem 1.2]{ChDeHe}.

A result due to Iwasawa \cite[Theorem 2]{iwasawa1951topological} yields that a
connected topological group $G$ is an [IN]-group if and only if the
topological commutator of $G$ is compact.

A nonabelian group to which this applies is the reduced Weyl-Heisenberg group,
which is the quotient of the simply connected, connected Heisenberg group by a
central discrete subgroup. More generally, Lemma \ref{[IN]case} also implies
that no step-two nilpotent Lie group with compact center is FT.
\end{remark}

It is currently open which groups have the property that inverses of
relatively separated sets are relatively separated again. As the previous
remark shows, some nonabelian nilpotent Lie groups do. However, \emph{simply
connected} nilpotent Lie groups generally do not:

\begin{lemma}
\label{lem:nilp_inv_relsep} If $G$ is a nonabelian, simply connected connected
nilpotent Lie group, then there exists a separated set $\Gamma\subset G$ such
that $\Gamma^{-1}$ is not relatively separated.
\end{lemma}

\begin{proof}
We start out by considering the special case that $G$ is the simply connected
and connected three-dimensional Heisenberg Lie group, with Lie algebra spanned
by
\[
X_{1}=\left[
\begin{array}
[c]{ccc}%
0 & 1 & 0\\
0 & 0 & 0\\
0 & 0 & 0
\end{array}
\right]  ,X_{2}=\left[
\begin{array}
[c]{ccc}%
0 & 0 & 0\\
0 & 0 & 1\\
0 & 0 & 0
\end{array}
\right]  ,X_{3}=\left[
\begin{array}
[c]{ccc}%
0 & 0 & 1\\
0 & 0 & 0\\
0 & 0 & 0
\end{array}
\right]
\]
with non-trivial Lie brackets $\left[  X_{1},X_{2}\right]  =X_{3}.$ It is easy
to verify that
\[
\exp\left(  X\right)  \exp\left(  Y\right)  =\exp\left(  X+Y+\frac{1}%
{2}\left[  X,Y\right]  \right)  =\exp\left[
\begin{array}
[c]{ccc}%
0 & x_{1}+y_{1} & x_{3}+y_{3}+\frac{x_{1}y_{2}-x_{2}y_{1}}{2}\\
0 & 0 & x_{2}+y_{2}\\
0 & 0 & 0
\end{array}
\right]  .
\]
Next, we endow the Heisenberg Lie group with the following quasi-norm
\[
\left\Vert \exp\left[
\begin{array}
[c]{ccc}%
0 & x_{1} & x_{3}\\
0 & 0 & x_{2}\\
0 & 0 & 0
\end{array}
\right]  \right\Vert =\left(  \left(  x_{1}^{2}+x_{2}^{2}\right)  ^{2}%
+x_{3}^{4}\right)  ^{1/4}.
\]
This induces a left-invariant quasi-metric between two arbitrary elements by
\begin{align*}
&  d\left(  \exp\left(  x_{1}X_{1}+x_{2}X_{2}+x_{3}X_{3}\right)  ,\exp\left(
y_{1}X_{1}+y_{2}X_{2}+y_{3}X_{3}\right)  \right) \\
&  =\left\Vert \exp\left(  -x_{1}X_{1}-x_{2}X_{2}-x_{3}X_{3}\right)
\exp\left(  y_{1}X_{1}+y_{2}X_{2}+y_{3}X_{3}\right)  \right\Vert \\
&  =\left(  \left(  \left(  y_{1}-x_{1}\right)  ^{2}+\left(  y_{2}%
-x_{2}\right)  ^{2}\right)  ^{2}+\left(  y_{3}-x_{3}+\frac{x_{2}y_{1}%
-x_{1}y_{2}}{2}\right)  ^{4}\right)  ^{1/4}.
\end{align*}
For any natural number $N\in\mathbb{N}$, we define%
\[
U_{N}=\left[
\begin{array}
[c]{ccc}%
0 & N^{2} & 0\\
0 & 0 & 0\\
0 & 0 & 0
\end{array}
\right]  ,V_{N,\ell}=\left[
\begin{array}
[c]{ccc}%
0 & N^{2} & \frac{\ell}{2}\\
0 & 0 & \frac{\ell}{N^{2}}\\
0 & 0 & 0
\end{array}
\right]  ,\ell=1,\cdots,N.
\]
Let
\[
\Gamma=\left\{  \exp(-V_{M,\ell}):M\in\mathbb{N},\ell=1,\cdots,M\right\}  .
\]
Note that, for all $N \in\mathbb{N}$,
\[
d\left(  \exp U_{N},\exp V_{N,\ell}\right)  =\frac{\ell}{N^{2}}\leq\frac{1}%
{N}.
\]
In other words, the ball with center $U_{N}$ and radius $N^{-1}$ contains at
least $N$ elements of $\Gamma^{-1}$, which shows that $\Gamma^{-1}$ is not
relatively separated. On the other hand, it is easy to verify for distinct
$\exp\left(  V_{N,\kappa}\right)  ,\exp\left(  V_{M,\ell}\right)  \in\Gamma,$
that
\[
d\left(  \exp\left(  -V_{N,\kappa}\right)  ,\exp\left(  -V_{M,\ell}\right)
\right)  \geq1.
\]
Thus, $\Gamma$ is left-uniformly discrete.\newline Now, if $G$ is a simply
connected, connected nonabelian nilpotent Lie group of dimension larger than
three, then by Kirillov's lemma \cite{Corwin}, $G$ contains a closed subgroup
$H$ that is isomorphic to the simply connected Heisenberg group. By the case
of the three-dimensional Heisenberg Lie group, there exists a separated (in
$H$) $\Gamma\subset H$ with the property that $\Gamma^{-1}$ is not separated
(in $H$). However, for subsets of $H$, being separated in $H$ is the same as
being separated in $G$.
\end{proof}

A further interesting property of FT groups is a kind of universal sampling
theorem. First some terminology:

\begin{definition}
Given a unitary representation $\pi$ of a locally compact group $G$ acting on
a Hilbert space $\mathcal{H}_{\pi}$ and $\eta\in\mathcal{H}_{\pi}$, we define
the associated wavelet transform $V_{\eta}:\mathcal{H}_{\pi}\rightarrow
C_{b}(G)$ as
\[
V_{\eta}u(x)=\langle u,\pi(x)\eta\rangle~.
\]
$\eta$ is called \textbf{admissible} if $V_{\eta}:\mathcal{H}_{\pi}\rightarrow
L^{2}(G)$ is well-defined and isometric. We call $\pi$ \textbf{strongly
square-integrable} if there exists an admissible vector $\eta\in
\mathcal{H}_{\pi}$.
\end{definition}

It is well-known that admissible vectors associated to a strongly
square-integrable representations give rise to weak-sense inversion formulae
\[
u = \int_{G} V_{\eta}u(x) \pi(x) \eta\, dx\,\,.
\]
We can now derive a discretization result for generalized wavelet transforms
over FT groups. Observe here that the sampling set is \textbf{universal},
i.e., it is picked independently of the representation whose inversion formula
it discretizes.

\begin{theorem}
\label{thm:universal_samp} Assume that $G$ has a frame $(\lambda_{G}(x)
\varphi)_{x \in\Gamma}$ of translates for $L^{2}(G)$. Then for every strongly
square-integrable representation $\pi,$ there exists $\eta\in\mathcal{H}_{\pi
}$ such that $(\pi(x) \eta)_{x \in\Gamma} $ is a frame.
\end{theorem}

\begin{proof}
Let $\psi\in\mathcal{H}_{\pi}$ be an admissible vector for $\pi$. Then
$V_{\psi}$ is a unitary equivalence between $\mathcal{H}_{\pi}$ and
$\mathcal{H}_{\pi,\psi} = V_{\psi}(\mathcal{H}_{\pi})$, and the latter is a
closed, left-invariant subspace of $L^{2}(G)$. The orthogonal projection onto
$\mathcal{H}_{\pi,\psi}$ is $P = V_{\psi}V_{\psi}^{*}$. It follows that $(P
\lambda_{G}(x) \varphi)_{x \in\Gamma}$ is a frame of $\mathcal{H}_{\pi,\psi}$.
Then, since $V_{\psi}^{*}: \mathcal{H}_{\pi,\psi} \to\mathcal{H}_{\pi}$ is a
unitary equivalence intertwining the group actions, we finally see that $\eta=
V_{\psi}^{*} \varphi$ is as required.
\end{proof}

\begin{remark}
There is an alternative proof available of the fact that the reduced
Weyl-Heisenberg group is not FT, which makes use of the Theorem
\ref{thm:universal_samp}, in combination with well-known necessary density
conditions for Gabor frames, as derived in \cite{ChDeHe}.
\end{remark}

For the remainder of this section we concentrate on discrete groups. For the
following results, recall that a \emph{Riesz basis} of a Hilbert space is a
frame $(\eta_{i})_{i \in I}$ satisfying the additional inequality
\[
\sum_{i \in I} |c_{i}|^{2} \le C \left\|  \sum_{i \in I} c_{i} \eta_{i}
\right\|  ^{2}\,\,.
\]

\begin{theorem}
Let $G$ denote a discrete group, $\Gamma\subset G$ and $\varphi\in\ell^{2}%
(G)$. Then the following are equivalent:

\begin{enumerate}
\item[(a)] $(\lambda_{G}(x))_{x \in\Gamma}$ is a frame of $\ell^{2}(G)$.

\item[(b)] $\Gamma= G$, and $(\lambda_{G}(x))_{x \in G}$ is Riesz basis of
$\ell^{2}(G)$.
\end{enumerate}
\end{theorem}

\begin{proof}
We only need to prove $(a) \Rightarrow(b)$. We first show that $\Gamma\subset
G$ is uniformly dense. We prove this by contradiction. Let $A$ denote the
lower frame bound, and pick $U \subset G$ finite such that
\[
\sum_{x \not \in U} |\varphi(x)|^{2} < A\,\,.
\]
The assumption that $\Gamma$ is not uniformly dense implies the existence of
$y \in G$ such that $y U \cap\Gamma$ is empty. We then get
\begin{align*}
A  &  \le\sum_{x \in\Gamma} \left|  \langle\mathbf{1}_{\{y\}} , \lambda_{G}(x)
\varphi\rangle\right|  ^{2}\\
&  = \sum_{x \in y^{-1} \Gamma} |\varphi(x)|^{2}\\
&  \le\sum_{x \not \in U} |\varphi(x)|^{2} < A
\end{align*}
by choice of $y$ and $A$, respectively. This is the desired contradiction.

Hence $\Gamma$ is uniformly dense, which means that $(\lambda_{G}(x)
\varphi)_{x \in G}$ is a frame as well. Thus the associated frame operator
$S_{\varphi}= V_{\varphi}^{*} V_{\varphi}: f \mapsto f \ast\varphi^{*}
\ast\varphi$ is a bounded, self-adjoint operator with bounded inverse,
commuting with left translations on $G$. Hence, letting $\eta= S_{\varphi
}^{-1/2} \varphi$ yields a tight frame generator, i.e., we have that $f
\mapsto f \ast\eta^{*}$ is an isometry, or equivalently, that
\[
f = V_{\eta}^{*} V_{\eta}f = f \ast\eta^{*} \ast\eta~,
\]
holds for all $f \in\ell^{2}(G)$. In particular
\[
\delta_{e} = \delta_{e} \ast\eta^{*} \ast\eta=\eta^{*} \ast\eta~,
\]
and thus $\| \eta\|^{2} = ( \eta^{*} \ast\eta)(e ) = 1$. We thus obtain that
$(\lambda_{G}(x) \eta)_{x \in G} \subset\ell^{2}(G)$ is a Parseval frame
consisting of unit vectors; and it is well known that such frames are actually
orthonormal bases.

But then $(\lambda_{G}(x)\varphi)_{x\in G}=(S_{\varphi}^{1/2}\lambda
_{G}(x)\eta)_{x\in G}$ is a Riesz basis, as the image of an orthonormal basis
under an invertible operator. In particular, the system $(\lambda
_{G}(x)\varphi)_{x\in\Lambda}$ is incomplete in $\ell^{2}(G)$, for every
proper subset $\Lambda$ of $G$. Thus $\Gamma= G$ follows.
\end{proof}

We conjecture that a group $G$ has a Riesz basis of translates if and only if
it is discrete.

\section{Sufficient criteria}

The results established so far do not seem to point towards the existence of
nondiscrete FT groups. As Proposition \ref{prop:phi_path} shows, the functions
$\varphi$ occurring in frames of translates are necessarily somewhat
pathological, an observation that seems to raise the bar somewhat further.
Nonetheless, the remainder of this paper will show that FT groups exist in
abundance. The strategy for proving such a result rests on two observations.
The first one is the following remarkably general discretization result,
recently established by Freeman and Speegle \cite[Theorem 1.3]{FS}:

\begin{theorem}
\label{thm:discretize} Let $(\eta_{x})_{x \in X} \subset\mathcal{H}$ denote a
family of bounded vectors in a separable Hilbert space, measurably indexed by
$x \in X$, where $(X,\mu)$ is a $\sigma$-finite measure space. Assume that
$(\eta_{x})_{x \in X}$ is a \textbf{continuous frame with respect to $\mu$},
i.e., there exist constants $0 < A \le B < \infty$ such that
\[
\forall g \in\mathcal{H}~:~A \| g \|^{2} \le\int_{X} |\langle g, \eta_{x}
\rangle|^{2} d\mu(x) \le B \| g \|^{2}~.
\]
Then there exists a countable family $(x_{i})_{i \in I}$ such that
$(\eta_{x_{i}})_{i \in I}$ is a frame.
\end{theorem}

The main consequence of this result is the following theorem which reveals a
first, large class of FT groups.

\begin{theorem}
\label{thm:typeI_unimod} Let $G$ be type I and non-unimodular. Then $G$ is an
FT group.
\end{theorem}

\begin{proof}
By \cite{Fu_02,Fu_LN}, $\lambda_{G}$ has an admissible vector $\eta$. I.e.,
the family $(\lambda_{G}(x) \eta)_{x \in G}$ is a continuous frame with
respect to $\mu_{G}$. Theorem \ref{thm:discretize} yields a family $(x_{i})_{i
\in I} \subset G$ such that $(\lambda_{G}(x_{i}) \eta)_{i \in I}$ is a frame.
\end{proof}

The second observation enlarges the class of FT further, by considering the
restriction of the regular representation to a suitable closed subgroup.

\begin{definition}
Given a representation $\pi$ of $G$, we say that $\pi$\textbf{\ has infinite
multiplicity} if $\pi\simeq\infty\cdot\pi.$
\end{definition}

\begin{remark}
\label{rem:inf_multiplicity} Let $G$ denote a type I group. Then the
Plancherel transform of $G$ gives rise to a unique direct integral
decomposition
\[
\lambda_{G} \simeq\int_{\widehat{G}} m_{\sigma}\cdot\sigma d\nu_{G}(\sigma)~,
\]
where $\nu_{G}$ is the \textbf{Plancherel measure} of $G$, and the
multiplicity $m_{\sigma}$ with which $\sigma\in\widehat{G}$ enters the
Plancherel decomposition is equal to the Hilbert space dimension of
$\mathcal{H}_{\sigma}$. In particular, $\lambda_{G}$ has infinite multiplicity
if and only if $m_{\sigma}= \infty$, for $\nu_{G}$-almost every $\sigma$.

Important classes of groups for which this holds are the nonunimodular type I
groups, and the nonabelian connected nilpotent Lie groups. In both cases,
$\nu_{G}$-almost all irreducible representations are induced from subgroups of
infinite index: For the nonunimodular case, this was established in
\cite{DuMo}; for the nilpotent case, it follows by Kirillov's orbit method,
see \cite{Corwin}. But the representation spaces associated to these
representations are then infinite-dimensional.

It is worthwhile noting that if a representation $\lambda_{G}$ has infinite
multiplicity, then
\[
\lambda_{G} \simeq m \cdot\lambda_{G}%
\]
holds as well, for all natural numbers $m$.
\end{remark}

The following result exploits an idea due to Iverson \cite[Theorem
3.6]{Iverson} for our purposes.

\begin{theorem}
\label{thm:subgt_ft} Let $H<G$ denote a closed subgroup that has FT, and such
that $\lambda_{H}$ has infinite multiplicity. Then $G$ is FT. In fact, there
exists a vector $\varphi\in L^{2}(G)$ and $\Gamma\subset H$ such that
$(\lambda_{G}(x) f)_{x \in\Gamma} \subset L^{2}(G)$ is a frame.
\end{theorem}

\begin{proof}
Using a measurable set of coset representatives $C$ mod $H$, we have a Borel
isomorphism $H \times C \ni(h,x) \mapsto hx \in G$. Furthermore, this
isomorphism intertwines the left action $H \times(H \times C) \ni(h_{1}%
,(h_{2},x)) \mapsto(h_{1} h_{2}, x) \in H \times C$ with the left action of
$H$ on $G $. By Weil's formula, we can then identify the Haar measure on $G$
with the product of Haar measure on $H$, and a suitable choice of measure on
the quotient. Since the latter is Borel isomorphic to $C$, this identification
induces a unitary equivalence $L^{2}(G) \to L^{2}(H) \otimes L^{2}(C)$.
Since the Borel isomorphism intertwines the respective actions on $H \times C$
and $G$, the unitary map intertwines $\lambda_{G}|_{H} $ with $\lambda_{H}
\otimes\mathbf{1}$, where $\mathbf{1}$ denotes the trivial representation
acting on $L^{2}(C)$.

Hence, denoting by $\kappa\in\mathbb{N}\cup\{\infty\}$ the Hilbert space
dimension of $L^{2}(C)$, we find
\[
\lambda_{G}|_{H}\simeq\kappa\cdot\lambda_{H}\simeq\lambda_{H}~.
\]
By assumption on $H$, there exists a frame of translates in $L^{2}(H)$, and
the image of this frame under the intertwining operator $\lambda_{H}%
\simeq\lambda_{G}|_{H}$ has the desired properties.
\end{proof}

The simplest example of a type I nonunimodular group is the $ax+b$-group, the
semidirect product $\mathbb{R}\rtimes\mathbb{R}^{+}$. Note that
$SL(2,\mathbb{R})$ contains the closed subgroup
\[
\left\{  \left[
\begin{array}
[c]{cc}%
a & b\\
0 & a^{-1}%
\end{array}
\right]  :a>0,b\in\mathbb{R}\right\}
\]
which is isomorphic to the $ax+b$ group. Hence Theorem \ref{thm:subgt_ft}
implies the following:

\begin{corollary}
$SL(2,\mathbb{R})$ is an FT group.
\end{corollary}

Note that $SL(2,\mathbb{R})$ is unimodular. Hence unimodular FT groups do
exist. Moreover, let $p,q$ be natural numbers satisfying $p+q>2$. Next, let
\[
\mathrm{SO}\left(  p,q\right)  =\left\{  M\in GL\left(  p+q,\mathbb{R}\right)
:M^{tr}J\left(  p,q\right)  M=J\left(  p,q\right)  \right\}
\]
where $J\left(  p,q\right)  =\mathrm{diag}\left(  1,\cdots,1,-1,\cdots
,-1\right)  $ such that the trace of $J\left(  p,q\right)  $ is equal to
$p-q.$ Then $\mathrm{SO}\left(  p,q\right)  $ is a closed subgroup of
$GL\left(  p+q,\mathbb{R}\right)  $ and the following holds true. 

\begin{corollary}
Given natural numbers $p,q$ such that $p+q>2,$ $\mathrm{SO}\left(  p,q\right)
$ is an FT group
\end{corollary}

\begin{proof}
We first observe that every element $X$ of the Lie algebra of $\mathrm{SO}%
\left(  p,q\right)  $ can be written in block-form as
\[
X=\left[
\begin{array}
[c]{cc}%
Z & S\\
S^{tr} & Y
\end{array}
\right]  \text{ for some }\left(  Z,Y\right)  \in\mathfrak{so}\left(
p\right)  \times\mathfrak{so}\left(  q\right)  .
\]
As such, if $q=1$ then%
\[
A=\left[
\begin{array}
[c]{cccccc}%
0 & 0 & \cdots & 0 & 0 & 1\\
0 & 0 & \cdots & 0 & 0 & 0\\
\vdots & \vdots & \ddots & \cdots & \vdots & \vdots\\
0 & 0 & \cdots & \ddots & 0 & 0\\
0 & 0 & \cdots & 0 & 0 & 0\\
1 & 0 & \cdots & 0 & 0 & 0
\end{array}
\right]  ,X=\left[
\begin{array}
[c]{cccccc}%
0 & 0 & \cdots & 0 & 1 & 0\\
0 & 0 & \cdots & 0 & 0 & 0\\
\vdots & \vdots & \ddots & \cdots & \vdots & \vdots\\
0 & 0 & \cdots & \ddots & 0 & 0\\
-1 & 0 & \cdots & 0 & 0 & 1\\
0 & 0 & \cdots & 0 & 1 & 0
\end{array}
\right]  \in\mathfrak{so}\left(  p,1\right)
\]
and $\left[  A,X\right]  =X.$ Thus, $\exp\left(  \mathbb{R}X\right)
\exp\left(  \mathbb{R}A\right)  $ is a closed nonunimodular subgroup of
$\mathrm{SO}\left(  p,1\right)  $ which is isomorphic to the $ax+b$-group and
its follows that $\mathrm{SO}\left(  p,1\right)  $ is FT. More generally, letting
\[
I=\left\{  \left(  1,p+1\right)  ,\left(  p+1,1\right)  \right\}  \text{ and
}J=\left\{  \left(  1,p\right)  ,\left(  p,p+1\right)  ,\left(  p+1,p\right)
\right\}
\]
and defining matrices $B,Y\in\mathfrak{gl}\left(  p+q,\mathbb{R}\right)$ with
entries satisfying
\[
B_{jk}=\left\{
\begin{array}
[c]{c}%
1\text{ if }\left(  j,k\right)  \in I\\
0\text{ otherwise}%
\end{array}
\right.  \text{ and }Y_{jk}=\left\{
\begin{array}
[c]{c}%
1\text{ if }\left(  j,k\right)  \in J\text{ }\\
-1\text{ if }\left(  j,k\right)  =\left(  p,1\right)  \\
0\text{ otherwise}%
\end{array}
\right.
\]
it is easy to verify that $\left[  B,Y\right]  =Y.$ Thus, the $ax+b$-group
$\exp\left(  \mathbb{R}Y\right)  \exp\left(  \mathbb{R}B\right)$ is a closed
subgroup of $\mathrm{SO}\left(  p,q\right)  $ and it follows immediately that
$\mathrm{SO}\left(  p,q\right)  $ is FT as well.
\end{proof}

\begin{remark}
Since $SL(n,\mathbb{R})$ and $SL(n,\mathbb{C})$, for $n\geq2$, contain closed
isomorphic copies of $SL(2,\mathbb{R})$, all of these groups are FT groups
again. The same reasoning applies to the symplectic groups $Sp(2n,\mathbb{R})$
and $Sp(2n,\mathbb{C})$: One has $Sp(2,\mathbb{R})=SL(2,\mathbb{R})$, and the
higher-dimensional groups contain closed isomorphic copies of $Sp(2,\mathbb{R}%
)$. Identical reasoning entails the FT property for the metaplectic groups in
arbitrary dimensions.
\end{remark}

\begin{remark}
We can use previous results to construct a rather unusual shearlet frame.
Consider the matrix group
\[
H=\left\{  \pm\left[
\begin{array}
[c]{cc}%
a & b\\
0 & a^{1/2}%
\end{array}
\right]  :a>0,b\in\mathbb{R}\right\}  ~.
\]
Let $G=\mathbb{R}^{2}\rtimes H$. The natural affine action of $G$ on
$\mathbb{R}^{2}$ gives rise to the quasi-regular representation on
$L^{2}(\mathbb{R}^{2})$, which is known to be strongly square-integrable
\cite{DaKuStTe}. The generalized wavelet transform associated with this group
is the so-called \textbf{shearlet transform}. Shearlet frames are typically
constructed by choosing a lattice $\Gamma\subset\mathbb{R}^{2}$ and a discrete
subset $H_{d}\subset H$ and considering families of the kind
\[
(\pi(hx,h)\psi)_{x\in\Gamma,h\in H_{d}}~,
\]
for suitable well-chosen functions $\psi$; see \cite{BeTa} for an early source
using this type of construction. Note that $H$ is a closed subgroup of $G$
that is isomorphic to the $ax+b$-group. Hence combining Theorems
\ref{thm:universal_samp} and \ref{thm:subgt_ft}, we can now show that there
exists a frame of the type
\[
(\pi(0,h)\varphi)_{h\in H_{d}}~,
\]
i.e., \emph{using only dilations!} This example generalizes easily to shearlet
groups in arbitrary dimension.
\end{remark}


\section{Lie groups with FT}

Using the tools developed in the previous sections, we will now explore the FT
property on a class of Lie groups known as exponential Lie groups. To present
the findings of this section, we need the following. Let $G$ be a Lie group
with Lie algebra $\mathfrak{g}.$

We say that $\mathfrak{g}$ is of \textbf{type R} if for every $X\in
\mathfrak{g,}$ the eigenvalues of the endomorphism $ad_{\mathfrak{g}}\left(
X\right)  $ are purely imaginary.

The \textbf{lower central series} of $\mathfrak{g}$ is inductively defined as
follows $C^{1}\mathfrak{g}=\mathfrak{g}$ and $C^{j+1}\mathfrak{g}=\left[
\mathfrak{g},C^{j}\mathfrak{g}\right]  $ for $j>0.$ Moreover, a Lie algebra is
\textbf{nilpotent} if there exists a natural number $k$ such that
$C^{k}\mathfrak{g}$ is trivial. According to Engel's characterization, a Lie
algebra $\mathfrak{g}$ is nilpotent if and only if for every $X\in
\mathfrak{g,}$ the endomorphism $ad_{\mathfrak{g}}\left(  X\right)
:Y\mapsto\left[  X,Y\right]  $ is nilpotent (see \cite{Hilgert}, Theorem
$5.2.8.$) The \textbf{derived series} of $\mathfrak{g}$ is a decreasing
collection $D^{j}\mathfrak{g}$ of ideals in $\mathfrak{g}$ defined as
$D^{0}\mathfrak{g}=\mathfrak{g}$ and $D^{j+1}\mathfrak{g}=\left[
D^{j}\mathfrak{g},D^{j}\mathfrak{g}\right]  $ for $j\in\mathbb{N}.$

A Lie algebra $\mathfrak{g}$ is \textbf{solvable} if its derived series
reaches the trivial algebra in finitely many steps. Suppose for now that
$\mathfrak{g}$ is a real solvable Lie algebra of dimension $n$. Then
$ad_{\mathfrak{g}}\left(  \mathfrak{g}\right)  $ is a solvable algebra of
endomorphisms acting on $\mathfrak{g.}$ Next, let $\mathfrak{g}_{\mathbb{C}}$
be the complexification of $\mathfrak{g.}$ As a $\mathfrak{g}_{\mathbb{C}}%
$-module, the Lie algebra $\mathfrak{g}_{\mathbb{C}}$ has a
\textbf{Jordan-H\"{o}lder} sequence
\[
\left\{  0\right\}  =\mathfrak{g}_{\mathbb{C}}^{\left(  0\right)  }%
\subset\mathfrak{g}_{\mathbb{C}}^{\left(  1\right)  }\cdots\subset
\mathfrak{g}_{\mathbb{C}}^{\left(  n-1\right)  }\subset\mathfrak{g}%
_{\mathbb{C}}^{\left(  n\right)  }=\mathfrak{g}_{\mathbb{C}}.
\]
That is, each $\mathfrak{g}_{\mathbb{C}}^{\left(  k\right)  }$ is an ideal and
$\dim\mathfrak{g}_{\mathbb{C}}^{\left(  k\right)  }=k.$ Moreover, the action
of $\mathfrak{g}_{\mathbb{C}}$ on the vector space $\mathfrak{g}_{\mathbb{C}%
}^{\left(  k\right)  }/\mathfrak{g}_{\mathbb{C}}^{\left(  k-1\right)  }$ where
$1\leq k\leq n$ defines a linear form on $\mathfrak{g}_{\mathbb{C}}$ called a
\textbf{root} of $\mathfrak{g.}$ We say that $\mathfrak{g}$ is an
\textbf{exponential solvable Lie algebra} if it does not possess any root with
nonzero purely imaginary value. In other words, each root of $\mathfrak{g}$ is
given by $X\mapsto\lambda\left(  X\right)  \left(  1+i\alpha\right)  $ where
$\lambda\in\mathfrak{g}^{\ast}$ and $\alpha\in\mathbb{R}.$

Next, let $G=\exp\mathfrak{g}$ be a simply connected and connected Lie group
with solvable Lie algebra $\mathfrak{g}.$ Then it is known that $\mathfrak{g}$
is exponential if and only if the exponential map determines an analytic
diffeomorphism between $\mathfrak{g}$ and $G.$ In this case, $G$ is called an
\textbf{exponential solvable Lie group}.

The following theorem summarizes our results

\begin{theorem}
\label{main} Exponential solvable Lie groups which are not nilpotent are FT.

\end{theorem}

It is worthwhile noting that for all FT groups mentioned in Theorem
\ref{main}, the set $\Gamma\subset G$ of shifts generating the frame can be
chosen as a subset of a closed subgroup of dimension at most $3$,
\emph{regardless of the dimension of $G$ itself.}

The case of nonabelian simply connected, connected nilpotent Lie groups is
currently open. There are several reasons why this class is interesting: A
comprehensive answer would ultimately settle the exponential solvable case.
Moreover, it would help clarify the relationship between the FT property and
the property that inverses of relatively separated sets are relatively
separated again. Recall that the latter implies the negation of the former,
and we currently have no example that the converse does not hold in general.
If such a general converse holds, it implies that all nonabelian simply
connected, connected nilpotent Lie groups are not FT, via Lemma
\ref{lem:nilp_inv_relsep}. However, with current knowledge, all we can do is
to reduce the discussion to certain test cases. For the following proposition,
we introduce $T(n,\mathbb{R})$ for the group of upper triangular matrices in
$\mathrm{GL}(n,\mathbb{R})$ with ones on the diagonal.

\begin{proposition}
\text{ }

\begin{enumerate}
\item[(a)] Every nonabelian simply connected, connected nilpotent Lie group is
FT if and only if the Heisenberg group is FT

\item[(b)] No nonabelian simply connected, connected nilpotent Lie group is FT
if and only if for all $n\geq3$, $T(n,\mathbb{R})$ is not FT.
\end{enumerate}
\end{proposition}

\begin{proof}
Both statements are consequences of Theorem \ref{thm:subgt_ft}, via the
observation that if $G$ is nonabelian and simply connected, nilpotent, there
exist embeddings
\[
H\subset G\subset T(n,\mathbb{R})
\]
as closed subgroups, where $H$ is isomorphic to a Heisenberg group, and $n$ is
sufficiently large. Here the first embedding is due to Kirillov's lemma
\cite{Corwin}, the second is a well-known consequence of Engel's Theorem.
\end{proof}


There are two solvable Lie algebras which are of crucial importance in this
work. The first one is the \textbf{ax+b} Lie algebra which is a
two-dimensional solvable Lie algebra spanned by $A,X$ such that $\left[
A,X\right]  =X.$ The second one: the \textbf{Gr\'{e}laud's algebra} is a
three-dimensional solvable Lie algebra spanned by $A,Y_{1},Y_{2}$ with
non-trivial Lie brackets
\[
\left[  A,Y_{1}\right]  =Y_{1}+\beta Y_{2}\text{ and }\left[  A,Y_{2}\right]
=-\beta Y_{1}+Y_{2}%
\]
for some nonzero real number $\beta.$

\begin{lemma}
\label{eigenLemma} If $\mathfrak{g}$ is not of type R then $\mathfrak{g}$
admits a Lie subalgebra $\mathfrak{h<g}$ which is either isomorphic to the
ax+b algebra or the Gr\'{e}laud's algebra.
\end{lemma}

\begin{proof}
By assumption, one of the following must hold. \newline Case $1$: There exists
$X\in\mathfrak{g}$ with a nonzero real eigenvalue $\lambda.$ As such, there
exists an eigenvector $Y$ for $ad\left(  X\right)  $ with corresponding
nonzero eigenvalue $\lambda$ satisfying $\left[  \lambda^{-1}X,Y\right]  =Y$
and $\mathfrak{h}=\mathbb{R}Y+\mathbb{R}X$ as desired. \newline Case $2$:
Suppose that Case $1$ does not hold. Then there exists $X\in\mathfrak{g}$ such
that $\alpha+i\beta$ is an eigenvalue for the endomorphism $ad(X)$ and
$\alpha\neq0.$ Thus, there exist $Y_{1},Y_{2}\in\mathfrak{g}$ such that
$\left[  X,Y_{1}\right]  =\alpha Y_{1}+\beta Y_{2}$ and $\left[
X,Y_{2}\right]  =-\beta Y_{1}+\alpha Y_{2}.$ Next, we claim that $Y_{1},Y_{2}$
must commute. Otherwise, a straightforward application of Jacobi's identity
yields $\left[  X,\left[  Y_{1},Y_{2}\right]  \right]  =2\alpha\left[
Y_{1},Y_{2}\right]  .$ However, this contradicts the fact that $X$ does not
have a nonzero real eigenvalue. Finally, setting $\mathfrak{h}=\mathbb{R}%
$-span$\left\{  X,Y_{1},Y_{2}\right\}  $ gives the desired result.
\end{proof}

\begin{lemma}
\label{solv1}Let $\mathfrak{g}$ be an $n$-dimensional exponential solvable Lie
algebra. Then the following statements are equivalent

\begin{enumerate}
\item $\mathfrak{g}$ is not a nilpotent Lie algebra

\item There exists a Lie subalgebra of $\mathfrak{h}$ of $\mathfrak{g}$ such
that $\exp\mathfrak{h}$ is a closed type I non-unimodular Lie group
\end{enumerate}
\end{lemma}

\begin{proof}
To prove that $(2)\Longrightarrow(1),$ suppose that there exists a Lie
subalgebra of $\mathfrak{h}$ of $\mathfrak{g}$ such that $\exp\mathfrak{h}$ is
a type I non-unimodular Lie group. Since nilpotent Lie groups are unimodular,
this subalgebra cannot be nilpotent. Then $\mathfrak{g}$ is not a nilpotent
Lie algebra since every subalgebra of a nilpotent Lie algebra is nilpotent
(see \cite{Corwin}, Proposition $1.1.6$). Conversely, let us suppose that
$\mathfrak{g}$ is an exponential solvable Lie group which is not nilpotent. By
Engel's Theorem, there exists $X\in\mathfrak{g}$ such that $ad_{\mathfrak{g}%
}(X)$ is not nilpotent. The fact that $\mathfrak{g}$ has no purely imaginary
roots together with Lemma \ref{eigenLemma} imply that $G$ admits a subalgebra
isomorphic to the ax+b Lie algebra or Grelaud's algebra. Since $G$ is
exponential, the associated Lie subgroup is simply connected and closed, and
since it is exponential solvable, it is of type I. Therefore
$(1)\Longrightarrow(2)$ holds.
\end{proof}

\subsection*{Proof of Theorem \ref{main}}

The proof of Theorem \ref{main} is a direct consequence of Lemma \ref{solv1}
and Theorem \ref{thm:subgt_ft}.

\bibliographystyle{plain}
\bibliography{frtr}

\end{document}